\documentclass[11pt]{amsart}
\usepackage{amsmath, amsthm, amssymb, amsfonts, enumerate, xspace}

\usepackage{xcolor}

\def \R {\mathbb{R}}

\def \Z {\mathbb{Z}}

\def \P {\mathbb{P}}

\def \LL {\mathcal{L}}

\def \< {\langle}
\def \> {\rangle}

\renewcommand\P{{\mathbb{P}}}

\newcommand\eps{{\varepsilon}}

\DeclareMathOperator{\supp}{supp}

\DeclareMathOperator{\Incomp}{Incomp}

\def \Dhat {\widehat{D}}

\theoremstyle{plain}
 \newtheorem{theorem}{Theorem}[section]

 \newtheorem{proposition}[theorem]{Proposition}
 
 \newtheorem{lemma}[theorem]{Lemma}
 \newtheorem{corollary}[theorem]{Corollary}

\newtheorem{remark}[theorem]{Remark}

\theoremstyle{definition}
\newtheorem{definition}[theorem]{Definition}

\begin{document}

\title[Sparse Simple Spectrum]{Sparse Random  matrices have simple spectrum}

\author[K. Luh]{Kyle Luh$^*$}
\address{$^*$John A. Paulson School of Engineering and Applied Sciences, Harvard University}
\email{kluh@seas.harvard.edu}
\thanks{$^*$K. Luh is supported by the National Science Foundation under Award
No. 1702533}

\author[V. Vu]{Van Vu${}^\dagger$}
\address{${}^\dagger$Department of Mathematics, Yale University, New Haven 06520}
\email{van.vu@yale.edu}
\thanks{${}^\dagger$V. Vu is supported by   NSF  grant DMS 1307797  and AFORS grant FA9550-12-1-0083.}

\date{}
\maketitle

\begin{abstract}
	Let $M_n$ be a class of symmetric sparse random matrices, with independent entries $M_{ij} = \delta_{ij} \xi_{ij}$ for $i \leq j$.  $\delta_{ij}$ are i.i.d. Bernoulli random variables taking the value $1$ with probability $p \geq n^{-1+\delta}$ for any constant $\delta > 0$ and $\xi_{ij}$ are i.i.d. centered,  subgaussian random variables.  We show that with high probability this class of random matrices has simple spectrum (i.e. the eigenvalues appear with multiplicity one).  We can slightly modify our proof to show that the adjacency matrix of a sparse Erd\H{o}s-R\'enyi graph has simple spectrum for $n^{-1+\delta } \leq p \leq 1- n^{-1+\delta}$.  These results are optimal in the exponent.  The result for graphs has connections to the notorious graph isomorphism problem.  
\end{abstract}

\section{Introduction}
The gaps between eigenvalues are natural objects to study in random matrix theory and are of central importance in the field.  Since the introduction of the notion of a random matrix, there have been numerous inquiries into the spacings of consecutive eigenvalues of symmetric random matrices.  For a matrix with eigenvalues $\lambda_i$, we denote the gaps by $\delta_i = \lambda_{i+1} - \lambda_i$.  
The limiting global gap distribution for Gaussian matrices (GOE and GUE) has been well understood for some time and can be deduced from Wigner's semi-circle law \cite[Chapter 6,7]{mehta2004random}.  Recent progress on universality has extended these results to large classes of random variables \cite{TaoVuLocalUniversality,ErdosYauGapUniversality}.  At finer levels, meaning under proper normalization and for a particular gap, the limiting distribution for the GUE was only calculated in 2013 by Tao \cite{tao2013asymptotic}.  The four moment condition establishes that this distribution is universal for any random variable that matches the gaussian up to the first four moments.  Using advanced dynamical techniques, Erd\H{o}s and Yau removed this condition \cite{erdos2012gap}.

Although these results describe the behavior of a single gap, $\delta_i$, bounds on the smallest gap, $\delta_{min} = \min_i \delta_i$, for general matrices were still out of reach.  Bourgade and Ben-Arous \cite{arous2013extreme} showed that $\delta_{min}$ is on the order of $n^{-5/6}$ for the GUE ensemble.  Yet, currently, this issue does not fall into the scope of the four moment theorem.  Although tail bounds for individual $\delta_i$ were known for more general matrices \cite{tao2011random,tao2010random}, they were too weak to survive the union bound over all $i$ to conclude anything about $\delta_{min}$.  Under severe restrictions on the smootheness and decay of the entries, Erd\H{o}s, Schlein and Yau \cite{ erdHos2009wegner} proved that
$$
\P\left (E n^{1/2} - \frac{\varepsilon}{n^{1/2}} \leq \lambda_i \leq \lambda_k \leq E n^{1/2} + \frac{\eps}{n^{1/2}} \text{ for some i} \right) = o(\eps^{k^2}).
$$         
for any fixed $k \geq 1$, any $\eps > 0$ and any bounded $E \in \mathbb{R}$.  Applying a union bound to this result yields
$$
\P(\delta_{min} \leq \delta n^{-1/2}) = o(n \delta^3) + \exp( -c n)
$$
for any $\delta > 0$.
Despite the strong bound, this result applies only to a small set of smooth random variables.  Outside of this set, even whether $\delta_{min}$ could equal zero could not be settled by these previous results and was only resolved in 2014.    Phrased differently, the fact that a random matrix typically has simple spectrum (i.e. all eigenvalues have multiplicity one) is a recent result due to Tao and Vu \cite{TaoVuSimple}.  They show that the probability that a random matrix has simple spectrum is bounded below by $1 - n^{-A}$ for any constant $A$.  In \cite{NguyenTaoVuGaps}, this qualitative statement was refined to quantitative tail bounds on the gaps between the eigenvalues and probability that a random matrix has simple spectrum was improved to $1 - \exp(- n^c)$ for a small unspecified constant $c$.     

In the realm of graphs, whether or not a graph has simple spectrum (i.e. its adjancecy matrix has simple spectrum) has practical complexity implications.  Although great strides have been made recently on the notorious graph isomorphism problem \cite{babai2016graph}, the best running time guarantees are still \emph{quasipolynomial}.   However, Babai, Grigoryev and Mount \cite{BabaiIsomorphism} demonstrated ealier that the graph isomorphism problem restricted to the graphs with simple spectrum is in complexity class $\mathcal{P}$.  A corollary of the random matrix result in \cite{TaoVuSimple} is that dense 
Erd{\H o}s-R{\'e}nyi random graphs have simple spectrum which answered a question of Babai's that had been open since the '80's.         

In the past few years, there has been renewed interest in sparse random matrices due to their applications in data science, where they require less storage space and fewer operations to manipulate \cite{DasguptaSparseJL, OSNAP,ClarksonWoodruff}.  In other settings, sparse random matrices reflect the intuition that in many natural problems, each data point is dependent on only a few of the many parameters \cite{CompressedSensing, FaceRecognition, ImageResolution,LuhVuDictionary}.  For random graphs, the more interesting behavior occurs for sparse graphs.  Many real-world networks are sparse and applications often prefer graphs with fewer edges that maintain the necessary properties.  

In this work, we establish that sparse random matrices have simple spectrum.  Our result is nearly optimal in terms of the range of sparsity. In the dense range, our work improves the probability bound in \cite{NguyenTaoVuGaps} to $1- \exp(-n^{1/128})$.  The particular value of the constant $(1/128)$ is not meaningful and has not been optimized.

\section{Main Results}
Let $M_n$ be an $n \times n$ symmetric random matrix with entries $m_{ij} = \delta_{ij} \xi_{ij}$  for all $i \leq j$, where $\delta_{ij}$ is a Bernoulli random variable that takes the value $1$ with probability $p = p(n)$ and $\xi_{ij}$ are iid random variables with mean zero, variance one, and subgaussian moment bounded by $B$.  Our main result is the following.
\begin{theorem}\label{thm:main}
For $0 < \delta \leq 1$ a constant and $p \geq n^{-1+\delta}$, then with probability at least $1-\exp(-(np)^{1/128})$, $M_n$ has simple spectrum.
\end{theorem}

Denote by $G(n,p)$, the random variable that takes values in the labeled graphs on $[n]$ vertices and distributed such that each edge appears independently with probability $p$.  
\begin{theorem}\label{thm:adjacencymatrix}
Let $A_n$ be the adjacency matrix of $G(n,p)$ and $0 < \delta \leq 1$ a constant.  For $n^{-1+\delta} \leq p \leq 1 - n^{-1+ \delta} $, with probability at least $1-\exp((np)^{-1/128})$, $A_n$ has simple spectrum.
\end{theorem}

\begin{remark}
Observe that for $p = o(\log n/n)$, there is likely to be at least two row of zeros in $M_n$ and $A_n$.  This yields a zero eigenvalue with multiplicity at least 2.  Thus, our bound on $p$ is near optimal. We record here that the upperbound on $p$ does not appear in Theorem \ref{thm:main} as even with $p=1$, there is additional randomness from the $\xi_{ij}$.  However, for the adjacency matrix, for $p=1$, we are left with the deterministic matrix $J_n - I_n$ which has eigenvalue $-1$ with multiplicity $n-1$.  By symmetry, the upperbound is also near optimal. In fact, we believe the true sparsity threshold is on the order of $p \geq \log n/n$, but our current method needs a technical refinement to achieve this bound and we postpone this matter for another occassion.
\end{remark}

The remainder of the paper is organized as follows.  In Section \ref{sec:proofstrategy} we give a birds-eye view of the proof, avoiding any technical statements.  
In Section \ref{sec:notation} we state several notational conveniences.  In Sections \ref{sec:compressible} and \ref{sec:incompressible} we develop the necessary tools to control the deviation of $M_n$ acting on two different sets of vectors (compressible and incompressible respectively).  Finally, in Section \ref{sec:mainproof}, we combine the results of the previous sections to obtain a short proof of Theorem \ref{thm:main}.  In the final section, we discuss the necessary modifications to handle adjacency matrices of sparse random graphs.     

\section{Proof Strategy} \label{sec:proofstrategy}
The overall approach is analogous to that used in \cite{TaoVuSimple} and \cite{NguyenTaoVuGaps}.
For $M_n$ as in Theorem \ref{thm:main}, we write
\begin{equation} \label{eq:decompositionMn}
M_n = 
\begin{pmatrix} 
  M_{n-1} & X \\ 
  X^T & m_{nn} 
\end{pmatrix},
\end{equation}
where $X = (x_1, \dots, x_{n-1}) \in \mathbb{R}^{n-1}$.  For a matrix $X$, let $\lambda_{n}(X) \leq \dots \leq \lambda_1(X)$ be the eigenvalues of $M_n$. Let $v = (x , a)$ (where $x \in \mathbb{R}^{n-1}$ and $a \in \mathbb{R}$) be the unit eigenvector associated to $\lambda_i(M_n)$. By definition we have
$$
\begin{pmatrix}
M_{n-1} & X \\ 
X^T & m_{nn}
\end{pmatrix}
\begin{pmatrix}
x \\
a
\end{pmatrix} = 
\lambda_i(M_n) \begin{pmatrix}
x \\
a
\end{pmatrix}.
$$
Restricting our attention to the top $n-1$ coordinates gives
$$
(M_{n-1} - \lambda_i(M_n)) x + a X = 0.
$$
Let $w$ be the eigenvector of $M_{n-1}$ corresponding to $\lambda_{i}(M_{n-1})$.  After multiplying on the right by $w^T$, we deduce that
$$
|a w^T X | = |w^T (M_{n-1} - \lambda_i(M_n)) x| = |\lambda_i(M_{n-1}) - \lambda_i(M_n)||w^T x|.
$$
By the Cauchy interlacing law, we must have $\lambda_i(M_n) \leq \lambda_i (M_{n-1}) \leq \lambda_{i-1} (M_n)$. 
Therefore, if we let $\mathcal{E}_i$ be the event that $\lambda_i(M_n) = \lambda_{i+1}(M_n)$, then assuming $a \neq 0$, on the event $\mathcal{E}_i$, this implies that $w^T X = 0$.  A simple union bound over all choices of $a$ in $w$ removes our assumption that $a \neq 0$.  Finally, if $\P(\mathcal{E}_i)  = o(n^{-1})$ for all $i$, then a union bound yields the result.

Our task thus reduces to showing that an eigenvector $w$ of $M_{n-1}$ has the property that $\P(w^T X = 0)$ is small.  Note that $X$ and $w$ are independent.  By now, this is a well-studied phenomenon \cite{TaoVuLO, NguyenVuLO, RudelsonVershyninLO}.  This small-ball probability is intimately related to the arithmetic structure of the vector $w$.  The goal then is to prove that with high probability, an eigenvector of the submatrix $M_{n-1}$ will not have much structure.  For this intermediate objective, we make the simple observation that for $v$, a unit eigenvector of $M_{n}$, with eigenvalue $\lambda$, 
$$
(M_n - \lambda) v = 0.
$$    
For $x$ close to $v$, $(M_n - \lambda) x \approx 0$.  This is reminiscent of the least singular problem for a random matrix and the details of our argument draws heavily from the techniques of \cite{RudelsonVershyninLO, VershyninInvertibility, BasakRudelsonInvertibility}.  Choosing an appropriate net of the highly structured vectors in $\mathcal{S}^{n-1}$ and a net of potential eigenvalues, we show that these vectors are unlikely to be eigenvectors.  

This aerial view of the argument obscures the technical obstacles that must be overcome when the matrices we deal with are sparse.  As the random variables are zero with large probability, the small-ball probabilities that appear tend to be too large for direct union bounds to work.  Delicate nets and careful balancing of probabilities is required to implement our overall strategy.

\section{Notation} \label{sec:notation}
For a vector $v \in \mathbb{R}^n$ and an index set $I \subset [n]$, let $v_I \in \mathbb{R}^{|I|}$ be the restriction of $v$ onto that index set and $P_I(v) \in \mathbb{R}^n$ be the vector $v$ with all coordinates in $I^c$ zeroed out.  

We will also need finer control over index sets $I \subset [n]$. We let $ord(I)$ be the vector in $\mathbb{N}^{|I|}$ populated by elements of $I$ in increasing order.  Then, we define $I[k] := ord(I)_k$ and $I[k:k'] := ord(I)_{[k:k']}$, where $[k:k'] := \{i: k \leq i \leq k'\}$.  

To avoid repition, we impose the assumption that, unless explicitly stated, any constants (usually numbered) in the statement of the Lemmas, Propositions and Theorems depend only on $\delta$ and the subgaussian moment $B$.  Additionally, standard asymptotic notation (e.g. $o$, $O$) is stated with the assumption of $n$ tending to infinity.

\section{Compressible Vectors} \label{sec:compressible}
\subsection{Preliminaries}
We divide the unit sphere into two classes.  The \emph{compressible} vectors are those that are close to sparse vectors and the remaining vectors are called \emph{incompressible}.
\begin{definition}
For $M \in \mathbb{N}$, a vector $x$ is in $\text{Sparse}(M)$ if $|\text{Supp}(x)| \leq M$.  For a $\delta \in (0,1)$, we denote 
$$
\text{Comp}(M, \delta) := \{x \in \mathcal{S}^{n-1} : \exists y \in \text{Sparse}(M) \text{ such that } \|x - y\|_2 \leq \delta \}.
$$
The incompressible vectors are defined to be 
$$\text{Incomp}(M, \delta):= \{ x \in \mathcal{S}^{n-1}: x \notin \text{Comp}(M, \delta) \}.$$
\end{definition}

We will often make use of the following bound on the operator norm of $M_n$. 
\begin{proposition} \label{prop:opnorm}
There exist constants $K, c_{\ref{prop:opnorm}} > 0$ such that
$$
\P(\|M_n\| \geq K \sqrt{p n} ) \leq \exp(-c_{\ref{prop:opnorm}} p n).
$$
\end{proposition}
\begin{proof}
	The proof of Theorem 1.7 in \cite{BasakRudelsonInvertibility} can easily be modified to handle symmetric random matrices.  
\end{proof}
\subsection{Compressible Vectors}

\begin{proposition}\label{prop:compressible}
There exist constants $C_{\ref{prop:compressible}}, \bar{C}_{\ref{prop:compressible}}, c_{\ref{prop:compressible}}, c'_{\ref{prop:compressible}} > 0$ depending only on $B$, such that for 
$$
p \geq \frac{C_{\ref{thm:main}} \log n}{n}, \quad  \ell_0 := \left \lceil \frac{\log 1/(8p)}{\log \sqrt{pn}} \right \rceil \text{ and } \lambda \in [-C_{\ref{prop:opnorm}} \sqrt{pn}, C_{\ref{prop:opnorm}} \sqrt{pn}].  
$$
we have
$$
\P(\exists x \in  \text{Comp}(M, \rho) \text{ s.t. } \|(M_n - \lambda) x\|_2 \leq C_{\ref{prop:compressible}} \rho \sqrt{pn}) \leq \exp(-c'_{\ref{prop:compressible}} pn)
$$
where $\rho := (\bar{C}_{\ref{prop:compressible}})^{-\ell_0 - 6}$ and $p^{-1} \leq M \leq c_{\ref{prop:compressible}}n$.
\end{proposition}

\begin{remark}
To gain some understanding of these parameters, observe that for $p = n^{-1 + \delta}$ for some constant $\delta > 0$, then $\ell_0 = O(1)$.  Near the threshold, when $p = \frac{\log n}{n}$, $\ell_0 =  \Theta(\log n / \log \log n)$ so $\rho = \exp(- O(\log n/ \log \log n))$.
\end{remark}

Although this result is highly non-trivial, the proof follows from a straightforward adaptation of Proposition 3.1 in \cite{BasakRudelsonInvertibility}.  We include the proof with the necessary modifications in Appendix \ref{sec:proofofcomp}.

From this result, we obtain a bound on the probability that an eigenvector is compressible. 
\begin{corollary}\label{corollary:comp}
For $p^{-1} \leq M \leq c_{\ref{prop:compressible}}n$, there exists a constant $c_{\ref{corollary:comp}} > 0$ such that
$$
\P( \exists \text{ a unit eigenvector } \in \text{Comp}(M, \rho)) \leq \exp( - c_{\ref{corollary:comp}} p n).
$$

\end{corollary}

\begin{proof}
By Lemma \ref{prop:opnorm}, all eigenvalues of $M_n$ are in the interval $I = [-K \sqrt{pn}, K \sqrt{pn}]$.  Consider an $n^{-1}$-net of $I$ which can be constructed to be of size at most $2K n \sqrt{pn}$.  For $\lambda \in I$ that is an eigenvalue of $M_n$ with eigenvector $x \in \text{Comp}(M,\rho)$, there exists an element of the net, $\lambda_0$, such that
$$
\|(M_n - \lambda_0) x\|_2 = \|(\lambda - \lambda_0) x\|_2 \leq n^{-1}. 
$$
However, by Proposition \ref{prop:compressible}, 
$$
\P(\exists x \in  \text{Comp}(M, \rho) \text{ s.t. } \|(M_n - \lambda_0) x\|_2 \leq n^{-1}) \leq \exp(-c'_{\ref{prop:compressible}} pn).
$$
Taking a union bound over the $\lambda_0$ and increasing $C_{\ref{thm:main}}$ if necessary, yields the result.
\end{proof}

\section{Incompressible Vectors} \label{sec:incompressible}
For these vectors, we develop small-ball probability bounds that are dependent on a measure of arithmetic strucutre (Least Common Denominator)\cite{RudelsonVershyninLO,VershyninInvertibility}.  First, we introduce the following partition of the indices for $v \in \text{Incomp}(M, \rho)$.  Let 
$$
\sigma(v):= \left \{ k: \frac{\rho}{\sqrt{2n}} \leq |v_k| \leq \frac{1}{\sqrt{M}} \right \}.
$$
Due to the incompressibility of $v$, the cardinality of this set is large.

\begin{lemma}\label{lemma:spreadset}
For $v \in \text{Incomp}(M, \rho)$,
$$
|\sigma(v)| \geq \frac{M \rho^2 }{2}.
$$

\end{lemma}
\begin{proof}
Define $\sigma_1(v) := \{k : |v_k| \leq \frac{1}{\sqrt{M}} \}$.  Since $v$ is a unit vector, $|\sigma_1^c| \leq M$.  As $y  = P_{\sigma_1^c} v$ is a sparse vector with support at most $M$, the definition of incompressible vectors implies $\|v-y\|_2 > \rho$ or $\|P_{\sigma_1}(v)\|_2^2 \geq \rho^2$.  Define the following set to capture the lower bound.
$$
\sigma_2(v) := \left \{k : |v_k| \geq \frac{\rho}{\sqrt{2n}} \right \}.
$$
Clearly, $\|P_{\sigma_2}(v) \|_2^2 \leq \rho^2/2.$  Therefore,
$$
\|P_\sigma(v)\|_2^2 \geq \|P_{\sigma_1}(v) \|_2^2 - \|P_{\sigma_2^c} (v) \|_2^2 \geq \rho^2 /2.
$$
By the upperbound on the coordinates in $\sigma$, 
$$
\frac{\rho^2}{2} \leq \|P_\sigma (v)\|_2^2 \leq \frac{|\sigma|}{M}.
$$
Rearranging this inequality finishes the proof.
\end{proof}
 
For every $v \in \text{Incomp}(M, \rho)$, we fix a set $\sigma(v)$ of size exactly $\lceil M \rho^2/2 \rceil$.
Let $\tau'(v)$ be the index set of the $M$ coordinates with largest magnitude.  If this set is not uniquely defined, choose one arbitrarily.  Let $\tau(v) := \tau'(v) \setminus \sigma(v)$ and $\bar{\sigma} := [n] \setminus (\tau \cup \sigma)$.  Now we divide $[n] \setminus \tau$ into disjoint sets $I_1, I_2, \dots, I_{k_0} \text{ and } J$, with $|I_k| = \lceil \alpha n\rceil \leq M$ for $1 \leq k \leq k_0$ and $|J| \leq \lceil \alpha n \rceil$ where $\alpha = o(1)$ is a parameter to be chosen later.  For $1 \leq k \leq k_0,$ we let 
$$
I_k   := \sigma(v)_{[(k-1) \left \lceil \frac{M \rho^2}{2k_0} \right \rceil +1 : k \left \lceil  \frac{M \rho^2}{2k_0} \right \rceil]} \cup \bar{\sigma}(v)_{[(k-1)\lceil |\bar{\sigma}|/k_0 \rceil +1: k \lceil |\bar{\sigma}|/k_0 \rceil ]}.
$$
Finally, let $I_0 := J \cup \tau$ so $|I_0| \leq 2M$ by our assumption on $\lceil \alpha n \rceil$.  In words, $I_0$ is the index set of the large coordinates and the leftover smaller coordinates.
Additionally, we have
$$
 \frac{1}{2 \alpha} \leq \frac{n - |\tau|}{\lceil \alpha n \rceil} - 1\leq k_0 \leq \frac{n - |\tau|}{\lceil \alpha n \rceil} \leq \frac{1}{\alpha}.
$$

The purpose of this construction is to have substantial control over the $\ell_2$ norm and the $\ell_{\infty}$ norm of each $v_{I_k}$ for $1 \leq k \leq k_0$.  In particular, we have
\begin{equation} \label{eq:boundsvI}
\|v_{I_k} \|_2 \geq \sqrt{\frac{M \rho^2 \alpha}{8} \frac{\rho^2}{2n}} = \frac{\rho^2}{4}\sqrt{\frac{M \alpha }{n}} := \rho' .
\end{equation}
Furthermore,
$$
\|v_{I_k} \|_ \infty \leq \frac{1}{\sqrt{M}} \text{ and } \|v_{I_k} \|_2 \leq 2 \sqrt{\frac{\alpha n}{M}}.
$$

The $I_k$'s are filled by drawing sequentially from $\sigma$ and $\bar{\sigma}$ so that the entire partition is determined by $\tau'$ and $\sigma$.  Thus, there are at most $\binom{n}{M} \binom{n}{M \rho^2/2}$ partitions for all the vectors in $\text{Incomp}(M,\rho)$.
\subsection{Small-Ball Probability}
Recall from the proof strategy in Section \ref{sec:proofstrategy} that we have reduced the problem to bounding the probability that an eigenvector of a random matrix is orthogonal to a random vector.  As we will use various epsilon-net approximations, we need a more quantitative version of orthogonality.  In particular, we need to bound the probability that the dot product of the eigenvector and the random vector are small.  This leads naturally to the notion of small-ball probability.    
 
\begin{definition}
	The \emph{L\'evy concentration} of a random vector $Z \in \mathbb{R}^n$ is defined to be
	$$
	\mathcal{L}(Z, \eps) = \sup_{u \in \mathbb{R}^n} \P(\|Z- u\|_2 \leq \eps)
	$$
\end{definition}
Intuitively, the structure of a vector, $v$, will highly affect the L\'evy concentration of the random variable $v \cdot X$ where $X$ is a random vector.  To formalize this concept, we begin with a measure for the arithmetic structure of an entire unit vector.  

\begin{definition}\label{def:LCD}
We define the least common denominator (LCD) of $x \in \mathcal{S}^{n-1}$ as
$$
D(x) = \inf\{\theta > 0 : dist(\theta x, \Z^n) < (\delta_0 p)^{-1/2}  \sqrt{\log_{+}(\sqrt{\delta_0 p}\theta)}\}
$$
where $\delta_0$ is an appropriate constant (see Remark \ref{remark:delta} below).  
This particular form of the LCD was first used in \cite{VershyninInvertibility}.
\end{definition}

\begin{remark}\label{remark:delta}
	There exists constants $\delta_0, \bar{\eps}_0 \in (0,1)$ such that for any $\eps \leq \bar{\eps}_0$, $\mathcal{L}(\delta \xi, \eps) \leq 1-\delta_0 p$ where $\P(\delta = 1) = p$ and $\xi$ is a subgaussian random variable with unit variance.  We fix such a $\delta_0$ in Definition \ref{def:LCD}.
\end{remark}

The quantitative relationship between the arithmetic structure of a vector and small ball probability is captured in the following proposition.  

\begin{proposition}[Proposition 4.2, \cite{BasakRudelsonInvertibility}] \label{prop:smallballprobability}
	Let $X \in \mathbb{R}^n$ be a random vector with i.i.d. coordinates of the form $\xi_j \delta_j$, where $\P(\delta_j = 1) = p$ and $\xi_j$'s are random variables with unit variance and finite fourth moment.  Then for any $v \in \mathcal{S}^{n-1}$, 
	$$
	\mathcal{L} \left( X \cdot v, \sqrt{p} \eps \right) \leq C_{\ref{prop:smallballprobability}} \left( \eps + \frac{1}{\sqrt{p} D(v)} \right)
	$$
	where $C_{\ref{prop:smallballprobability}}$ depends only on the fourth moment of $\xi$.
\end{proposition}

Following \cite{VershyninInvertibility}, we introduce a tool that can reveal the arithmetic structure in small segments of the vector $x$.

\begin{definition}[Regularized LCD]					\label{def reg LCD}
  Let $\alpha \in (0, 1)$. 
  We define the {\em regularized LCD} of a vector $v \in \Incomp(M,\rho)$ as
  $$
  \Dhat(v,\alpha) = \max_{1 \leq j \leq k_0}  D \big(x_{I_j}/\|x_{I_j}\|_2\big).
  $$
\end{definition}

Combining Proposition \ref{prop:smallballprobability} with the by now standard tensorization argument (see \cite{RudelsonVershyninLO}), yields a bound on the L\'evy concentration of $M_n x$.

\begin{proposition}[Small ball probabilities of $M_n x$ via regularized LCD] \label{prop:smallballprob}
There exists a constant $C_{\ref{prop:smallballprob}}$ such that for all $\eps \geq 0$, and $I$ is an index set of size $\lceil \alpha n \rceil$, 
$$
  \LL(M_n x, \eps \|v_I\|_2 \sqrt{p m}) \leq C_{\ref{prop:smallballprob}}^{n- \lceil \alpha n \rceil} \left( \eps + \frac{1}{\sqrt p D(v_I/\|v_I\|_2)} \right)^{n- \lceil \alpha n \rceil}
$$
\end{proposition}

Therefore, by the bounds in (\ref{eq:boundsvI}) and the above proposition, we have
$$
\LL(M_n x, \eps \rho' \sqrt{p n}) \leq  \left( C_{\ref{prop:smallballprob}} \eps + \frac{C_{\ref{prop:smallballprob}}}{\sqrt p \hat{D}(v, \alpha)} \right)^{n - \alpha n} .
$$

We also have the following simple lower bound for the LCD.
\begin{proposition}[Lemma 6.2, \cite{VershyninInvertibility}] \label{prop:LCDlowerbound}
Let $x \in \mathcal{S}^{n-1}$.  Then 
$$
D(x) \geq \frac{1}{2 \|x\|_\infty}.
$$
\end{proposition}

We can deduce from this proposition and our bounds in (\ref{eq:boundsvI}), that
\begin{equation} \label{eq:LCDlowerbound}
\hat{D}(v, \alpha) \geq \frac{1}{2} \rho' \sqrt{M}.
\end{equation}

For the remainder of this section, we fix several parameters.  For the readers' convenience, we have aggregated and highlighted several important variables below.  Although we repeat these definitions, we urge the reader to refer to this section 
when verifying calculations later.   

\begin{center}
\boxed{
$$
M = \frac{n}{(np)^{1/16}}, \quad \alpha = (np)^{-1/16},  \quad \rho' = \frac{\rho^2}{4}\sqrt{\frac{M \alpha}{n}}
$$
}
\end{center}

\begin{remark}
	Due to the assumption that $p \geq n^{-1+\delta}$, we have that 
	$$
	c_\delta \leq \rho \leq c'_\delta.
	$$
	for two constants $c_{\delta}, c'_{\delta}$ only depending on $\delta$.  We will often implicitly make use of the fact that $np \rightarrow \infty$.
\end{remark}

\subsection{Vectors with Mid-Range and Small LCD}
In this section, we show that matrices of the form $M_n - \lambda$ are unlikely to have vectors in their nullspace with mid-range or small LCD.

\subsubsection{Mid-range LCD: $\frac{1}{\bar{c}} \frac{n^{1/2}}{(pn)^{1/32}} \leq \hat{D} \leq \exp((np)^{1/32})$} 
\begin{proposition}[Mid-range LCD]\label{prop:lowLCD}
For $\delta > 0$, $p \geq n^{-1 + \delta}$ and $\lambda \in [-K \sqrt{pn}, K \sqrt{pn}]$.  There exist constants  $$c_{\ref{prop:lowLCD}},  c''_{\ref{prop:lowLCD}},  \tilde{c}_{\ref{prop:lowLCD}}>0$$  such that for $M =  \frac{n}{(np)^{1/16}}$, 
$$
\P \left(\exists v \in \hat{S}_D \text{ s.t. } \|(M_n - \lambda) v\|_2 \leq \tilde{c}_{\ref{prop:lowLCD}} \eps_0 (pn)^{7/16}   \right) \leq \exp(- c''_{\ref{prop:lowLCD}} n)
$$
where $\frac{1}{\bar{c}} \frac{n^{1/2}}{(pn)^{1/32}} \leq D \leq \exp((np)^{1/32})$, $\eps_0 = c_{\ref{prop:lowLCD}}  \frac{n^{1/2}}{(np)^{1/32} D}$ and 
$$
\hat{S}_D := \{v \in \text{Incomp}(M, \rho): D \leq \Dhat(v) \leq 2D \}.
$$
Recall that  $\rho := (C K)^{-\ell_0 - 6}$, $\ell_0 := \left \lceil \frac{\log 1/(8p)}{\log \sqrt{pn}} \right \rceil$. 
\end{proposition}

\subsubsection{Level Sets for the Usual LCD}
We first construct level nets of the LCD (not regularized) for vectors of length $\lceil \alpha n \rceil$.  We drop the ceiling function when such precision is not crucial.  We keep the $n$ dependence in this section as various parameters, e.g. $p(n)$, more conveniently depend on $n$ rather than $\alpha n$.
\begin{lemma} \label{lemma:LCDlevelsetnet}
For $ p \geq n^{-1 + \delta}$,
$$
\beta = \frac{2 \sqrt{\log( 2 \sqrt{\delta_0 p} D_0)}}{D_0 \sqrt{\delta_0 p}}
$$
and $D_0 > 0$, the set $\{v \in \mathcal{S}^{\alpha n-1}: D(v) \in (D_0, 2D_0]\}$ has a $\beta$-net, $\mathcal{N}$, such that 
$$|\mathcal{N}| \leq \left( 2 + \frac{\bar{c} D_0}{\sqrt{\alpha n}} \right)^{\alpha n}  $$
for a universal constant $\bar{c}$.
\end{lemma}

\begin{proof}

For a $v$ with $D(v) \in (D_0, 2D_0]$, by the definition of LCD, there exists a $\theta \in (D_0, 2D_0]$ and $z \in Z$ such that 
$$
\| \theta v - z \|_2 < \frac{\sqrt{\log( \sqrt{\delta_0 p} \theta)}}{\sqrt{\delta_0 p}}
$$  
which implies that
$$
\left \|v - \frac{z}{\theta} \right \|_2 < \frac{\sqrt{\log( 2 \sqrt{\delta_0 p} D_0)}}{D_0 \sqrt{\delta_0 p}}.
$$
We also have
$$
\Bigg \| \left \|\frac{z}{\|z\|_2} \right \|_2 - \frac{\|z\|_2}{\theta} \Bigg \|_2 = \left \|  \|v\|_2  - \frac{\|z\|_2}{\theta} \right \|_2  < \left \| v - \frac{z}{\theta} \right \|_2.
$$
Combining the above estimates gives
\begin{align*}
\left \|v - \frac{z}{\|z\|_2} \right \|_2 &\leq \left\| v - \frac{z}{\theta} \right \|_2 + \left \| \frac{z}{\theta} - \frac{z}{\|z\|_2} \right \|_2 \\
&\leq \left\| v - \frac{z}{\theta} \right \|_2 + \Bigg \| \left \|\frac{z}{\|z\|_2} \right \|_2 - \frac{\|z\|_2}{\theta} \Bigg \|_2 \\
&\leq \frac{2 \sqrt{\log( \sqrt{\delta_0 p} \theta)}}{D_0 \sqrt{\delta_0 p}}.
\end{align*}

Note that 
$$
\|z\|_2 \leq \|z - \theta v\|_2 + \| \theta v\|_2 \leq  \frac{\sqrt{\log( \sqrt{\delta_0 p} \theta)}}{\sqrt{\delta_0 p}} + 2D_0 \leq 4 D_0.
$$
The last inequality follows from recalling that $D_0 \geq \rho' \sqrt{M} \geq \frac{c_{\delta}^2 \sqrt{n}}{4 (np)^{5/32}} $ so $D_0 \sqrt{p} \geq \frac{c_{\delta}^2}{4} (np)^{11/32} \geq \sqrt{\log (\sqrt{\delta_0 p} \theta)}$.
Let
$$
Z := \{z \in \mathbb{Z}^{m}: \text{supp}(z) \in I \text{ and } 0 < \|z\|_2 \leq 4D_0 \}.
$$
Define $\mathcal{N}:= \{z/ \|z\|_2: z \in Z \}$.  By the standard volumetric calculation,
$$
|\mathcal{N}| \leq \left( 2 + \frac{\bar{c} D_0}{\sqrt{m}} \right)^{m}
$$
for some universal constant $\bar{c}$. 
$\mathcal{N}$ serves as an appropriate net.  
\end{proof}

The above lemma can be modified so that $\beta$ is a function of $D$ rather than $D_0$.

\begin{lemma}
For 
$$
\beta = \frac{2 \sqrt{\log( 2 \sqrt{\delta_0 p} D)}}{D \sqrt{\delta_0 p}}
$$
and $D_0 > 0$, the set $\{v \in \mathcal{S}^{\alpha n-1}: D(v) \in (D_0, 2D_0]\}$ has a $\beta$-net, $\mathcal{N}$, such that 
$$|\mathcal{N}| \leq \left( 12 + \frac{\bar{c} D}{\sqrt{\alpha n}} \right)^{\alpha n}  $$
for a universal constant $\bar{c}$.
\end{lemma}

\begin{proof}
By Lemma \ref{lemma:LCDlevelsetnet}, the set is covered by at most $\left( 2 + \frac{\bar{c} D}{\sqrt{\alpha n}} \right)^{\alpha n}$ balls of radius $\beta_0 = \frac{2 \sqrt{\log( 2 \sqrt{\delta_0 p} D_0)}}{D_0 \sqrt{\delta_0 p}}$.  If $\beta \geq \beta_0$ then the result follows immediately.  Assume $\beta < \beta_0$.  A $\beta/2$ net of size $(4 \beta_0/ \beta)^{\alpha n} \leq (3D/ D_0)^{\alpha n} $.  Therefore, the number of small balls is at most
$$
\left( 2 + \frac{\bar{c} D_0}{\sqrt{\alpha n}} \right)^{\alpha n} \left(\frac{3D}{D_0} \right)^{\alpha n}  \leq \left(12 + \frac{\bar{c} D_0}{\sqrt{\alpha n}} \right)^{\alpha n}
$$
\end{proof}

Now we extend the net to cover all vectors with LCD less than $2D_0$.
\begin{lemma} \label{lemma:netsoflcd}
For $D >  f(n) = \omega(1)$, then the set 
$$
\{v \in \mathcal{S}^{\alpha n-1} : f(n) \leq D(v) \leq D \} 
$$  
has a $\beta$-net of size at most
$$
\left(12 + \frac{\bar{c} D_0}{\sqrt{\alpha n}} \right)^{\alpha n} \log(D)
$$ 
\end{lemma}
\begin{proof}

Decompose the set
$$
\{x \in \mathcal{S}^{m-1} : D(v) \leq D \} = \bigcup_k \{v \in \mathcal{S}^{m-1}: D(x) \in (2^{-k}D, 2^{-k+1} D] \}
$$
where the union is over all $k$ such that $(2^{-k}D, 2^{-k+1}D]$ has non-zero intersection with $[f(n), D]$.  Each of these intervals has a $\beta$-net by Lemma \ref{lemma:LCDlevelsetnet}.  There are at most $\log D$ such $k$.
\end{proof}

\subsubsection{Nets for the Level Sets of the Regularized LCD}

Let 
$$
S_{\hat{D}} := \{v \in \text{Incomp}(M, \rho): D < \Dhat(v) \leq 2D \}
$$
and set
$$
 M = \frac{n}{(np)^{1/16}}, \text{  } \alpha = (np)^{-1/16} \text{ and } \eps_0 = \frac{c_{\ref{prop:lowLCD}} \sqrt{\alpha n}}{D}.
$$
We record several useful bounds which are consequences of our choice of parameters,
$$
\frac{c_{\delta}^2}{4} (np)^{-1/16} \leq \rho' \leq \frac{{c'_{\delta}}^2}{4} (np)^{-1/16},
$$

One can check that $p^{-1} \leq M$ since $np \rightarrow \infty$ so that $M$ is in the range of Corollary \ref{corollary:comp}.  

Let $c$ be a constant less than $1/ 2 C_{\ref{prop:smallballprob}}$.  We first create an $c \rho' \eps_0/10 K$-net for the coordinates in $I_0$.  For this set, we use a trivial-net $\mathcal{N}_0$ of size at most 
$$
\left( \frac{10 K}{c \eps_0 \rho'}\right)^{2M}.
$$
Recall that $\|v_{I_k} \|_ \infty \leq \frac{1}{\sqrt{M}}$, so by Porposition \ref{prop:LCDlowerbound}, $\hat{D}(v) \geq \rho' \sqrt{M}$ for any $v \in \text{Incomp}(M, \rho)$.
For each $I_k$ with $1 \leq k \leq k_0$, by Lemma \ref{lemma:netsoflcd} and the fact that $\text{LCD}(v_{I_k}/ \|v_{I_k} \|_2) \leq 2D$, we can create a $\beta = \frac{2 \sqrt{\log( 2 \sqrt{\delta_0 p} D)}}{D \sqrt{\delta_0 p}}$-net of size at most 
$$
\left(12 + \frac{\bar{c} D_0}{\sqrt{\alpha n}} \right)^{\alpha n} \log(D)
$$
Let $\bar{\mathcal{N}}_I$ be an $\frac{c \rho' \eps_0}{10 K k_0}$-net of $[\rho', 1]$ of size at most $\frac{30 K k_0}{c \eps_0 \rho'}$.

Define $$\mathcal{M} := \{x + \sum_k t_k y_k: x \in \mathcal{N}_0, y \in \mathcal{N}_k, t \in \bar{\mathcal{N}}_k \}.$$
We note that 
$$
|\mathcal{M}| \leq \left( \frac{10 K}{c \eps_0 \rho'}\right)^{2M} \prod_k \left(12 + \frac{\bar{c} D_0}{\sqrt{\lambda n}} \right)^{\lambda n} \log(D) \left(\frac{30 K k_0}{c \eps_0 \rho'} \right)
$$
For any $v \in S_D$, there exists a $m = x + \sum_k t_k y_k \in \mathcal{M}$ such that
$$
\|x - v_{I_0} \|_2 \leq \frac{c \rho' \eps_0}{10 K},  \quad \left \|y_k - \frac{v_{I_k}}{\|v_{I_k}\|_2} \right \|_2 \leq \beta, \quad \text{ and } \big |t_k - \|v_{I_k} \|_2 \big | \leq \frac{c \rho' \eps_0}{10 K k_0}. 
$$
Therefore,
\begin{align*}
\|v - m\|_2 &\leq \frac{c \rho' \eps_0}{10 K} + \sum_k \left(\Big \|v_{I_k} - \|v_{I_k} \|_2 y_k \Big \|_2 + \Big \|\|v_{I_k} \|_2 y_k - t_k y_k \Big \|_2 \right) \\
&\leq \frac{c \rho' \eps_0}{10 K} + \sum_k \left(\left \|\frac{v_{I_k}}{\|v_{I_k} \|_2} -  y_k \right \|_2 \|v_{I_k}\|_2 + \Big \|\|v_{I_k} \|_2 y_k - t_k y_k \Big \|_2 \right) \\
&\leq \frac{c \rho' \eps_0}{10 K} + k_0 \left( 2  \beta  + \frac{c \rho' \eps_0}{10 K k_0 } \right) \\
&\leq \frac{c \rho' \eps_0}{5 K} +  \frac{1}{\alpha} \frac{2 \sqrt{\log( 2 \sqrt{\delta_0 p} D)}}{D \sqrt{\delta_0 p}} \\
&= \frac{c \rho' \eps_0}{5 K} + \frac{4}{\sqrt{M} \rho^2} \frac{1}{\alpha} \frac{2 \sqrt{\log( 2 \sqrt{\delta_0 p} D)}}{c_{\ref{prop:lowLCD}} \sqrt{\delta_0 p}} \rho' \eps_0
 \end{align*}
Using the upper bound on $D \leq \exp\left(  (np)^{1/32} \right) \leq \exp\left( \frac{c_{\ref{prop:lowLCD}}^2 M \rho^4 \alpha^2 p}{K^2}\right) $,  we deduce that
$$
\|v-m\|_2 \leq \frac{c \rho' \eps_0}{5 K} +  o(\rho' \eps_0) \leq \frac{\rho' \eps_0}{4 K}.
$$
At this point, there is no guarantee that the elements of $\mathcal{M}$ lie in $\hat{S}_D$.  We rectify this issue by slightly adjusting $\mathcal{M}$.  For every $m \in \mathcal{M}$, if there exists a $v \in \hat{S}_D$ such that $\|v - m\|_2 \leq \frac{c \rho' \eps_0}{4K}$ then replace $m$ by $v$.  Otherwise, simply discard $m$.  We call this new set $\mathcal{M}'$ and note that $|\mathcal{M}'| \leq |\mathcal{M}|$.  By the triangle inequality, $\mathcal{M}'$ is a $\frac{c \rho' \eps_0}{2K}$-net of $\hat{S}_D$.

\subsubsection{Proof of Proposition \ref{prop:lowLCD}}
\begin{proof}
Fix a $\lambda \in [-K, K]$.
In the last section, we showed that for all the vectors with the same $\sigma, \tau$, $\mathcal{M}'$ is an $c \rho' \eps_0/ 2K$-net of $\hat{S}_D$. 
Let $\mathcal{E}_{\mathcal{M}'}$ be the event that there exists a $m \in \mathcal{M}'$ such that $\|(M_n - \lambda ) m \|_2 > c \rho' \eps_0 \sqrt{pn}$.
As we fixed $c < 1/ 2 C_{\ref{prop:smallballprob}}$ and one can verify that 
$$
\eps_0 \geq \frac{1}{\sqrt{p} D},
$$   
by Lemma \ref{prop:smallballprob}, 
$$
\P\left(\mathcal{E}_{\mathcal{M}} \right) \leq |\mathcal{M}|  \eps_0^{n - \lceil \alpha n \rceil}.
$$

By our lower bound on $D$, we have that $\bar{c} D / \sqrt{\alpha n} \geq 1$. 
\begin{align*}
\P\left(\mathcal{E}_{\mathcal{M}'} \right) &\leq  \binom{n}{2M} \binom{n}{M \rho^2}\left( \frac{10 K}{c \eps_0 \rho'}\right)^{2M}  \left(\frac{2 \bar{c} D}{\sqrt{\alpha n}} \right)^{n} \log^{\alpha^{-1}}( 2D) \left(\frac{30 K k_0}{c \eps_0 \rho'} \right)^{\alpha^{-1}} \eps_0^{n - \lceil \alpha n \rceil} \\
&\leq \binom{n}{2M} \binom{n}{M \rho^2}\left( \frac{10 K}{ c \rho'}\right)^{2M}  \left(\frac{2 \bar{c} D}{\sqrt{\alpha n}} \right)^{n} \log^{\alpha^{-1}}(2D) \left(\frac{30 K k_0}{c \rho'} \right)^{\alpha^{-1}} \left( \frac{c_{\ref{prop:lowLCD}} \sqrt{\alpha n} }{D} \right)^{n- \lceil \alpha n \rceil - 2M - \alpha^{-1}} \\
&\leq \exp\Bigg(-n \bigg( -\frac{2 M}{n} \log n -  \frac{{c'}_{\delta}^2 M}{n} \log n - \frac{2M}{3n} \log(pn)  + \frac{n - \alpha n - 2M - \alpha^{-1}}{n} \log(1/ c_{\ref{prop:lowLCD}}) \\
&\quad \quad  - \log(2 \bar{c}) - \frac{\alpha^{-1}}{10 n} \log(pn) - \frac{\alpha^{-1}}{n} \log \left(\frac{120 K k_0}{c c_\delta^2} (pn)^{1/32} \right) - \frac{\alpha n + 2M + \alpha^{-1}}{n} \log(D/ \sqrt{\alpha n}) \bigg) \Bigg )\\
&\leq \exp\Bigg(-n \bigg(  -\log(2 \bar{c}) + \frac{1}{2} \log\left(\frac{1}{c_{\ref{prop:lowLCD}}}\right) - \frac{\alpha n + 2M + \alpha^{-1}}{n} (np)^{1/32} + o(1)   \bigg) \Bigg) \\ 
&\leq \exp \Bigg ( -n \bigg (  -\log(2 \bar{c}) + \frac{1}{2} \log\left(\frac{1}{c_{\ref{prop:lowLCD}}}\right) + o(1)   \bigg )\Bigg) \\
&\leq \exp(- c''_{\ref{prop:lowLCD}} n)
\end{align*}

for small enough $c_{\ref{prop:lowLCD}}$.  On the event $\overline{\mathcal{E}_{\mathcal{M}'}}$, for any $v \in \hat{S}_D$, we can find a $m \in \mathcal{M}'$ such that $\|v - m\|_2 \leq \rho' \eps_0/ 2K$.  Therefore,
$$
\|(M_n - \lambda) v \|_2 \geq \|(M_n - \lambda) m\|_2 - \|M_n - \lambda\| \|v - m\|_2 \geq  \frac{c \rho' \eps_0 \sqrt{pn}}{2}.
$$
The proof is complete upon setting $\tilde{c}_{\ref{prop:lowLCD}} = c c_{\delta}^2 c_{\ref{prop:lowLCD}}$. 
\end{proof}

\begin{remark}
	Note that a trivial $\eps_0$ net of the unit sphere is of size $(3D/c_{\ref{prop:lowLCD}} \sqrt{\alpha n})^n$ which is of the same order as our more involved construction.  However, the key gain of our design is that it is $\bar{c}$ that appears in the dominant term of our net size and $c_{\ref{prop:lowLCD}}$ can be defined independently.  
\end{remark}
\subsubsection{Small LCD}
  For this range of regularized LCD, a nearly identical argument as Proposition \ref{prop:lowLCD} applies.  As the choice of parameters is different, we show the necessary computations below.  
     
\begin{proposition}[Small LCD]\label{prop:verylowLCD}
	For $\delta > 0$, $p \geq n^{-1 + \delta}$ and $\lambda \in [-K \sqrt{pn}, K \sqrt{pn}]$.  There exist constants  $$c_{\ref{prop:verylowLCD}}, c''_{\ref{prop:verylowLCD}},  \tilde{c}_{\ref{prop:verylowLCD}}>0$$  such that for $M =  \frac{n}{(np)^{1/8}}$ 
	$$
	\P \left(\exists v \in \hat{S}_D \text{ s.t. } \|(M_n - \lambda) v\|_2 \leq \tilde{c}_{\ref{prop:verylowLCD}} \eps_0 (pn)^{7/16}   \right) \leq \exp(- c''_{\ref{prop:verylowLCD}} n)
	$$
	where $\rho' \sqrt{M} \leq D \leq \frac{1}{\bar{c}} \sqrt{\alpha n}$, $\eps_0' = (\rho' \sqrt{pM})^{-1/2} $ and 
	$$
	\hat{S}_D := \{v \in \text{Incomp}(M, \rho): D \leq \Dhat(v) \leq 2D \}.
	$$
	Recall that  $\rho := (C K)^{-\ell_0 - 6}$, $\ell_0 := \frac{\log 1/(8p)}{\log \sqrt{pn}}$. 
\end{proposition}

Using this new $\eps'_0$, we have

\begin{align*}
\|v - m\|_2 &\leq \frac{c \rho' \eps_0}{10 K} + \sum_k \left(\Big \|v_{I_k} - \|v_{I_k} \|_2 y_k \Big \|_2 + \Big \|\|v_{I_k} \|_2 y_k - t_k y_k \Big \|_2 \right) \\
&\leq \frac{c \rho' \eps_0'}{10 K} + \sum_k \left(\left \|\frac{v_{I_k}}{\|v_{I_k} \|_2} -  y_k \right \|_2 \|v_{I_k}\|_2 + \Big \|\|v_{I_k} \|_2 y_k - t_k y_k \Big \|_2 \right) \\
&\leq \frac{c \rho' \eps_0'}{10 K} + k_0 \left( \beta + \frac{c \rho' \eps_0'}{10 K k_0} \right) \\
&\leq \frac{c \rho'  \eps_0'}{5 K } +  k_0 \frac{2 \sqrt{\log( 2 \sqrt{\delta_0 p} \rho' \sqrt{M})}}{\rho' \sqrt{M} \sqrt{\delta_0 p} } \\
&\leq \frac{c \rho'  \eps_0'}{5 K } +  k_0 \frac{2 \sqrt{\log( 2 \sqrt{\delta_0 p} \rho' \sqrt{M})}}{(\rho') \rho'^{1/2}  (Mp)^{1/4} \sqrt{\delta_0 }} \rho' \eps_0' \\
&\leq \frac{c \rho' \eps_0'}{10 K}
 \end{align*}

The third to last inequality follows from the observation that the function $x \rightarrow \frac{\sqrt{ \log(c_1 x)}}{c_2 x}$ is a decreasing function for large values of $x$, $\sqrt{\delta_0 p} \rho' \sqrt{M} \geq (np)^{3/8}$ and $np  \rightarrow \infty$.  The last inequality follows from the simple calculation $\rho' \rho'^{1/2} (Mp)^{1/4} \rightarrow \infty$.  

Again, it is easy to check that $\eps_0 \geq \frac{1}{\sqrt{p} D}$, so
by Lemma \ref{prop:smallballprob}, 
$$
\P\left(\mathcal{E}_{\mathcal{M}} \right) \leq |\mathcal{M}|  \eps_0^{n - \lceil \alpha n \rceil}.
$$

As, $\bar{c} D / \sqrt{\alpha n} < 1$, 
\begin{align*}
\P\left(\mathcal{E}_{\mathcal{M}} \right) &\leq \binom{n}{2M} \binom{n}{M \rho^2}\left( \frac{10 K}{c \eps_0 \rho'}\right)^{2M}  \left(13\right)^{n} \log^{\alpha^{-1}}(2D) \left(\frac{30 K k_0}{c \eps_0 \rho'} \right)^{\alpha^{-1}} \eps_0^{n - \lceil \alpha n \rceil} \\
&\leq \binom{n}{2M} \binom{n}{M \rho^2}\left( \frac{10 K}{c  \rho'}\right)^{2M}  \left(13 \right)^{n} \log^{\alpha^{-1}}(2D) \left(\frac{30 K k_0}{c \rho'} \right)^{\alpha^{-1}} \left( \frac{1}{\sqrt{ \rho'} (pM)^{1/4}} \right)^{n- \lceil \alpha n \rceil - 2M - \alpha^{-1}} \\
&\leq \exp(- n)
\end{align*}
We extend this to all the vectors in $\hat{S}_D$ by the same approximation argument.

\section{Proof of Theorem \ref{thm:main}} \label{sec:mainproof}
\begin{proof}
By Corollary \ref{corollary:comp}, with probability $1 - \exp(-c_{\ref{corollary:comp}} n)$, the eigenvectors of $M_{n-1}$ are not compressible.  We now show that the eigenvectors of $M_{n-1}$ do not have mid-range or small regularized LCD.  We begin with the mid-range vectors.  Let $\frac{1}{\bar{c}} \frac{n^{1/2}}{(pn)^{1/32}} \leq D \leq \exp((np)^{1/32})$.  We demonstrate that an eigenvector is unlikely to be in $\hat{S}_D$.   
Let $\mathcal{P}$ be a $\tilde{c}_{\ref{prop:lowLCD}} \eps_0 (pn)^{7/16}$-net of $[-K \sqrt{pn}, K \sqrt{pn}]$ with 
$$
|\mathcal{P}| \leq \frac{2K \sqrt{pn}}{\tilde{c}_{\ref{prop:lowLCD}} \eps_0 (pn)^{7/16}} \leq \exp((np)^{1/16}). 
$$	
If $v \in \hat{S}_D$ and is an eigenvector with eigenvalue $\lambda$, then there is a $\lambda_0 \in \mathcal{P}$ with $|\lambda - \lambda_0| \leq \tilde{c}_{\ref{prop:lowLCD}} \eps_0 (pn)^{7/16}$.  Therefore,
$$
\|(M_n - \lambda_0)v \|_2 \leq |\lambda - \lambda_0| \leq \tilde{c}_{\ref{prop:lowLCD}} \eps_0 (pn)^{7/16}. 
$$ 
Thus, by Proposition \ref{prop:lowLCD} and a union bound, with probability greater than $1 - \exp(-c''_{\ref{prop:lowLCD}} n/2)$, an eigenvector of $M_{n-1}$ will not lie in $\hat{S}_D$.  Consider the following decomposition.  
\begin{align*}
\{v \in Incomp(M, \rho) &: \frac{1}{\bar{c}} \frac{n^{1/2}}{(pn)^{1/32}} \leq \hat{D}(v) \leq \exp((np)^{1/32}) \} = \\
& \bigcup_k \{v \in Incomp(M, \rho): D(v) \in (2^{-k}\exp((np)^{1/32}), 2^{-k+1} \exp((np)^{1/32})] \}
\end{align*}
where $k$ takes values such that $(2^{-k}\exp((np)^{1/32}), 2^{-k+1} \exp((np)^{1/32})] $ has a non-zero intersection with $[\frac{1}{\bar{c}} \frac{n^{1/2}}{(pn)^{1/32}},\exp((np)^{1/32}]$.  There are at most $(np)^{1/16}$ such $k$, so by a simple union bound, we can guarantee that the event, $\mathcal{E}_{mid}$, that $M_{n-1}$ does not have eigenvectors with regularized LCD in $[\frac{1}{\bar{c}} \frac{n^{1/2}}{(pn)^{1/32}},\exp((np)^{1/32}]$ occurs with  probability at least $1 - exp(-c''_{\ref{prop:lowLCD}}/3)$.  By an identical argument, replacing  Proposition \ref{prop:lowLCD} with Proposition \ref{prop:verylowLCD}, we have that the event, $\mathcal{E}_{small}$, that   the eigenvectors of $M_{n-1}$ have regularized LCD that lie outside of the interval $[\rho' \sqrt{M}, \frac{1}{\bar{c}} \frac{n^{1/2}}{(pn)^{1/32}} ]$ occurs with probability at least $1 - \exp(-c_{\ref{prop:verylowLCD}} n/3)$.  On the event $\mathcal{E}_{mid} \cap \mathcal{E}_{small}$, by Proposition \ref{prop:smallballprobability}, 
$$
\P(\mathcal{E}_i| \mathcal{E}_{mid} \cap \mathcal{E}_{small}) \leq \P(X \cdot v = 0 |\mathcal{E}_{mid} \cap \mathcal{E}_{small}) \leq \frac{C_{\ref{prop:smallballprobability}}}{\sqrt{p} \exp((np)^{1/32})} \leq \exp(-(np)^{1/64})
$$ 
where $X$ is from the decomposition of $M_n$ in (\ref{eq:decompositionMn}).  Taking  a union bound over all $i$, we conclude that $M_n$ has simple spectrum with probability at least $1 - \exp(-(np)^{1/64}/2)$.

\end{proof}

\section{Erd\H{o}s-R\'enyi Random Graphs}\label{section:randomgraph}
Let $G_n$ be a random variable that takes values in the simple graphs on $n$ vertices with vertex set $[n]$.  $G_n$ is distributed such that an edge appears between two vertices independently with probability $p$.  Let $A_n$ denote the adjacency matrix of $G_n$, i.e. 
$$
A_{ij}  = \left\{\begin{array}{ll}
1 & \mbox{if a directed edge is present between $i$ and $j$ in $G_n$}\\
0 & \mbox{otherwise}.\end{array} \right.
$$
Note that the entries of $A_n$ have mean $p$.  Thus, Theorem \ref{thm:main} does not immediately apply.  However, the expected adjancecy matrix is $p(J - I)$ where $J$ is the $n \times n$ all ones matrix.  However, $J$ is a rank one matrix and we can exploit this fact to adjust our proof to handle this case.  As the proof is only slightly modified, we do not repeat the argument and only highlight the necessary changes.  These adjustments follow those in \cite[Section 7]{BasakRudelsonInvertibility}.  

In preparation for the proof of Theorem \ref{thm:adjacencymatrix}, we need analogues of Proposition \ref{prop:compressible}, Proposition \ref{prop:lowLCD} and Proposition \ref{prop:verylowLCD}.  The necessary changes for Proposition \ref{prop:compressible} are discussed in Appendix \ref{appendix:noncentered}.  For Propositions \ref{prop:lowLCD} and Proposition \ref{prop:verylowLCD}, the first step was to obtain estimates on the L\'evy concentration.  As this function is insensitive to shifts in the mean, the first part of the proof holds without change.  For the net arguments to hold, we simply make the observation that $A_n - p (J_n - I_n)$ is a mean zero random matrix so the standard arguments (e.g. those in Proposition \ref{prop:opnorm}) yield
$$
\P(\|A_n - p(J_n - I_n) \|_2 \geq K' \sqrt{pn}) \leq \exp(c' p n).
$$

Therefore we have the following two propositions.

\begin{proposition}[Mid-range LCD]\label{prop:lowLCDAdjacency}
	For $\delta > 0$, $n^{-1+\delta} \leq p \leq 1/2$ and $\lambda \in [-K \sqrt{pn}, K \sqrt{pn}]$.  There exist constants  $$c_{\ref{prop:lowLCDAdjacency}}, c''_{\ref{prop:lowLCDAdjacency}},  \tilde{c}_{\ref{prop:lowLCDAdjacency}}>0$$  such that for $M =  \frac{n}{(np)^{1/16}}$, 
	$$
	\P \left(\exists v \in \hat{S}_D \text{ s.t. } \|(A_n - p(J_n - I_n) - \lambda) v\|_2 \leq \tilde{c}_{\ref{prop:lowLCDAdjacency}} \eps_0 (pn)^{7/16}    \right) \leq \exp(- c''_{\ref{prop:lowLCDAdjacency}} n)
	$$
	where $\frac{1}{\bar{c}} \frac{n^{1/2}}{(np)^{1/32}} \leq D \leq \exp( (pn)^{1/32})$, $\eps_0 = c_{\ref{prop:lowLCDAdjacency}}  n^{1/2 - \delta/12} / D$ and 
	$$
	\hat{S}_D := \{v \in \text{Incomp}(M, \rho): D \leq \Dhat(v) \leq 2D \}.
	$$

\end{proposition}

\begin{proposition}[Small LCD]\label{prop:verylowLCDAdjacency}
	For $\delta > 0$, $n^{-1+\delta} \leq p \leq 1/2$ and $\lambda \in [-K \sqrt{pn}, K \sqrt{pn}]$.  There exist constants  $$c_{\ref{prop:verylowLCDAdjacency}}, c''_{\ref{prop:verylowLCDAdjacency}}, \tilde{c}_{\ref{prop:verylowLCDAdjacency}}>0$$  such that for $M =  \frac{n}{(np)^{1/8}}$ 
	$$
	\P \left(\exists v \in \hat{S}_D \text{ s.t. } \|(M_n - p(J_n - I_n) - \lambda) v\|_2 \leq \frac{\tilde{c}_{\ref{prop:verylowLCDAdjacency}}}{2} \eps_0 (pn)^{7/16}   \right) \leq \exp(- c''_{\ref{prop:verylowLCDAdjacency}} n)
	$$
	where $\rho' \sqrt{M} \leq D \leq \frac{1}{\bar{c}} \sqrt{\alpha n}$, $\eps_0' = (\rho' \sqrt{pM})^{-1/2} $ and 
	$$
	\hat{S}_D := \{v \in \text{Incomp}(M, \rho): D \leq \Dhat(v) \leq 2D \}.
	$$
	
\end{proposition}
  
\subsection{Proof of Theorem \ref{thm:adjacencymatrix}}
\begin{proof}(Sketch)
	We first handle the case where $p \leq 1/2$.  Observe that the set $\{J_n x : x \in \mathcal{S}^{n-1}\} =  \{\theta \cdot \mathbf{1}: \theta \in [-n, n]\}$ where $\mathbf{1}_n$ is the vector of all ones.  Let $\mathcal{X}_n = \{\kappa \cdot \mathbf{1}: \kappa \in [-pn, pn]\}$. As this is a one-dimensional set, we can create a net with small cardinality.  Let $\mathcal{B}$ be a $\tilde{c}_{\ref{prop:lowLCDAdjacency}} \eps_0 (pn)^{7/16} $-net of $\mathcal{X}_n$ with 
	$$|\mathcal{B}| \leq \frac{2 pn}{\tilde{c}_{\ref{prop:lowLCDAdjacency}} \eps_0 (pn)^{7/16}} \leq \exp((np)^{1/16}).$$  

	By the triangle inequality, for $x, x' \in \mathcal{X}_n$,
	$$
	\Big |\inf \| (A_n - p(J_n - I_n) - \lambda) v - x\|_2 - \| (A_n - p(J_n - I_n) - \lambda) v - x'\|_2 \Big | \leq |x - x'|
	$$
	Therefore, the standard union bound and triangle inequality argument shows that for $D$ in the mid-range LCD,
	$$
	\P(\inf_{x \in \mathcal{X}_n} \inf_{v \in S_D} \|(A_n - p(J_n - I_n) - \lambda) v - y\|_2 \leq \tilde{c}_{\ref{prop:lowLCDAdjacency}} \eps_0 (pn)^{7/16} ) \leq \exp(- c n).
	$$
	The same applies for the low-range.  Finally observing that,
	$$
	\inf_{x \in \mathcal{X}_n} \inf_{v \in S_D} \|(A_n - p(J_n - I_n) - \lambda) v - y\|_2 \leq \inf_{v \in S_D} \|(A_n - (\lambda -p) )v\|_2,
    $$
    summing over the level sets as before and using the same net argument on $\lambda$, we can conclude that any eigenvector of $A_n$ has large LCD.  The rest of the argument proceeds as in the proof of Theorem \ref{thm:main}.
    
    For the remaining $p > 1/2$ case, we observe that the adjacency matrix of $G(n,p)$, $A_n(p)$, has the same distribution as $J_n - I_n - A_n(1-p)$.  Therefore, to control $\|(A_n(p) - p(J_n - I_n) \lambda)v\|_2$ it suffices to manage $\|(A_n(1-p) - (1-p)(J_n - I_n))v \|_2$, for which our previous argument applies.   
\end{proof}

\section{Concluding Remarks}
As mentioned before, we believe the threshold for a random matrix to have simple spectrum should be $p \sim \log n/ n$ rather than $p \sim n^{-1+\delta}$.  The calculations near the threshold are more involved and will appear elsewhere.  Additionally, our arguments naturally offer a quantitative bound on the size of the gaps between eigenvalues and the smallest absolute value of an eigenvalue (which is needed to bound the condition number of the matrix).  We have made no attempt to optimize these bounds so we pursue this line of work in a separate article.  

The proof of our result for adjacency matrices applies almost without change to matrices of the form $R_n + M_n$ where $R_n$ is a deterministic low-rank matrix.  However, to generalize this result to arbitrary non-zero mean matrices requires several new tools which we are currently developing.  The $\varepsilon$-net arguments that lie at the core of our current work fail in this setting as we no longer have the necessary control on the operator norm of the matrix and the image of the matrix may not be a perturbation of a low-dimensional space as for the adjacency matrix.  To address these new concerns, it will be necessary to use sparse versions of the Inverse Littlewood-Offord theorems of the second author and Nguyen.      

\bibliographystyle{plain}
\bibliography{Sparse}

\appendix 
\section{Proof of Proposition \ref{prop:compressible}}
\subsection{Matrix Lemma} \label{sec:proofofcomp}
The following observation was first utilized in \cite{BasakRudelsonInvertibility}.  If we fix the submatrix corresponding to the support of a sparse vector, it is likely that many of these rows will contain exactly one non-zero entry.  In this case, in the product of the matrix with the sparse vector, there is no cancellation in these coordinates.  As we are dealing with deviations of a matrix from a fixed vector $u$, we simply modify the lemma to show that there are many rows with exactly one non-zero coordinate with a convenient sign.  

\begin{lemma}\label{lemma:matrix}
	Let $M_n$ be a $n \times n$ matrix with independent entries $m_{ij} = \delta_{ij} \xi_{ij}$ where $\delta_{ij}$ are Bernoulli random variables with $\P(\delta_{ij}=1 ) = p$, where $p \geq C_{\ref{thm:main}}\log n / n$ and $\xi_{ij}$ are iid random variables with $\max\{\P(\xi_{ij} \geq 1), \P(\xi_{ij} \leq -1)\} \geq c_0$.  For $\kappa \in \mathbb{N}$, we define $\mathcal{E}^{J J'}_c$ to be the event that for any vector of signs $\{\varepsilon_j\}_{j=1}^n$ there are at least $c \kappa pn$ rows of $M_n$ for which there is exactly one non-zero entry $m_{ij}$ with $m_{ij} \eps_j \geq 1$ and $i \neq j$ in the columns corresponding to $J$, and all zero entries in the columns corresponding to $J'$.  Let 
	$$
	\mathfrak{m} = \mathfrak{\kappa} := \kappa \sqrt{pn} \wedge \frac{1}{8p}.
	$$
	Then, there exists constants $0 < c_{\ref{lemma:matrix}}, c'_{\ref{lemma:matrix}}$, depending only on $c_0$ such that
	
	\[
	\P\Bigg( \bigcap_{\kappa \le (8p \sqrt{pn})^{-1} \vee 1} \ \bigcap_{J \in \binom{[n]}{\kappa}} \ \bigcap_{J' \in \binom{[n]}{\mathfrak{m}}, \, J \cap J'=\emptyset} {\mathcal{E}}^{J, J'}_{c'_{\ref{lemma:matrix}}}\Bigg) \ge 1- \exp (-c_{\ref{lemma:matrix}} p n).
	\]
\end{lemma}

\begin{proof}
	The same proof as in \cite[Proof of Lemma 3.2]{BasakRudelsonInvertibility} yields the result when applied to the upper $\lfloor n/2 \rfloor \times \lfloor n/2 \rfloor$ submatrix of $M_n$.  The entries in this submatrix are independent.  
\end{proof}

\subsection{Very Sparse Vectors}
\begin{definition}
	For any $x \in S^{n-1}$, let $\pi_x: [n] \to [n]$ be a permutation which arranges the absolute values of the coordinates of $x$ in an non-increasing order. For $1 \le m \le m' \le n$, denote by $x_{[m:m']} \in \R^n$ the vector with coordinates
	\[
	x_{[m:m']}(j)=x(j) \cdot \mathbf{1}_{[m:m']}(\pi_x(j)).
	\]
	In other words, we include in $x_{[m:m']}$ the coordinates of $x$ which take places from $m$ to $m'$ in the non-increasing rearrangement.
	
	\noindent
	For $\alpha<1$ and $m \le n$ define the set of vectors with dominated tail as follows:
	\[
	\text{Dom}(m, \alpha):= \{ x \in S^{n-1} \mid\|x_{[m+1:n]}\|_2 \le \alpha \sqrt{m} \|x_{[m+1:n]}\|_{\infty} \}.
	\]
\end{definition}

\begin{lemma}\label{lemma:verysparse}
	Denote 
	$$
	\ell_0 := \left \lceil \frac{\log(1/ 8p)}{\log \sqrt{pn }} \right \rceil
	$$
	
	\begin{align*}
	\P\Big( \exists x \in &\text{Dom}(1/8p, (C_{\ref{lemma:verysparse}} K)^{-1}) \\ 
	&\qquad \text{ such that } \|(M_n  - \lambda) x  \|_2 \leq (C'_{\ref{lemma:verysparse}} K)^{-\ell_0} \sqrt{p n} \\
	&\qquad \text{ and } \|M_n - \lambda I\| \leq  K \sqrt{p n}\Big) \\
	&\leq \exp(-c_{\ref{lemma:verysparse}} pn)
	\end{align*}
\end{lemma}

\begin{proof}
We begin by diving $[n]$ into two roughly equal sets.  Let $n_0 = \lceil n/2 \rceil$.  We denote this decomposition by 
$$
M_n - \lambda I = 
\begin{pmatrix} 
A & B \\ 
B^T & C 
\end{pmatrix}, \quad
x = \begin{pmatrix} y \\ z \end{pmatrix}.
$$
Thus, we have the following equivalence.
$$
\|(M_n - \lambda) x \|_2^2 = \|Ay + Bz\|_2^2 + \|B^T y + C z \|_2^2.
$$
We condition on a realization of $A$ and $C$.
$$
\P(\exists x \in Sparse(m) \cap \mathcal{S}^{n-1} \text{ such that } \|(M_n - \lambda) x\|_2 \leq \sqrt{ c np} ) 
$$

Let us begin with the assumption that $p \geq (1/4)n^{-1/3}$.  In this regime, $\ell_0 = 1$.  For $k \in [n]$ let $J_k = \{k\}$ and $J'_k = \supp(x)\setminus J_k$.  Define the following vectors of signs.  $\{\eps_j\}_{j=1}^{n} = \{sgn(z_j) \cdot sgn((Ay)_j)\}_{j=1}^{n}$.
\begin{align*}
\|(M_n - \lambda) x\|_2^2 &\geq \sum_{k \in supp(x) \cap [1, n_0]} \sum_{i \in I_k} ((Bz)_i + (Ay)_i))^2 + \sum_{k \in supp(x) \cap [n_0 +1, n]} \sum_{i \in I_k} ((B^T y)_i + (Cz)_i)^2 \\
&\geq \sum_{k \in supp(x) \cap [1, n_0]} \sum_{i \in I_k} (Bz)_i^2 + \sum_{k \in supp(x) \cap [n_0 +1, n]} \sum_{i \in I_k} (B^T y)_i^2 \\
&\geq \sum_{k \in supp(x) \cap [1, n_0]} c_{\ref{lemma:matrix}} p n z_i^2 + \sum_{k \in supp(x) \cap [n_0 +1, n]} c_{\ref{lemma:matrix}} p n y_i^2 = c_{\ref{lemma:matrix}} p n,
\end{align*}
where in the final inequality we have invoked Lemma \ref{lemma:matrix} with the necessary signs.
Now we extend this estimate to $\text{Dom}(1/8p, (CK)^{-1})$.  Let $m = (8p)^{-1}$.  Assume that 
$$
\|(M_n - \lambda) x\|< \frac{1}{2} \sqrt{c_{\ref{lemma:matrix}} p n}
$$
Since $x \in \mathcal{S}^{n-1}$, we have $\|x_{[m+1:n]}\|_\infty \leq m^{-1/2}$.  Therefore,
$$
\|x_{[m+1:n]}\|_2 \leq (CK)^{-1} \sqrt{m} \|x_{[m+1:n]}\|_\infty \leq (CK)^{-1}.
$$
Therefore we have
\begin{align*}
\|(M_n - \lambda) x_{[1:m]} \|_2 \leq \|(M_n - \lambda) x\|_2 + (K \sqrt{p n}) (C K)^{-1} < \frac{3}{4} \sqrt{c_{\ref{lemma:matrix}} p n}
\end{align*}
for $C \geq \frac{4}{\sqrt{c_{\ref{lemma:matrix}}}}.$

Furthermore,
$$
\Big | \|(M_n - \lambda) (x_{[1:m]}/ \|x_{[1:m]})\|_2 - \|(M_n - \lambda)x_{[1:m]} \|_2 \Big| \leq 
K |1 - \|x_{[1:m]} \|_2 | \leq \frac{\sqrt{c_{\ref{lemma:matrix}} p n}}{2}
$$

Now we address the remaining $\frac{C_{\ref{thm:main}} \log n}{n} \leq p \leq (1/4) n^{-1/3}$.  Note that
$$
\frac{1}{8p \sqrt{pn}} > 1.
$$

Let $x \in Dom(1/8p, (CK)^{-1})$.  We rearrange the coordinates of $x$ by decreasing magnitude and group them into blocks of size $(pn)^{\ell/2}$ with $\ell = 1, \dots, \ell_0$.  From here on, for simplicity, we assume that $(pn)^{\ell_0/2} = 1/8p$.  In other words, set
$$
z_\ell = x_{[(pn)^{(\ell-1)/2} + 1: (pn)^{\ell/2}]},
$$  
and 
$$
z_{\ell_0+1} = x_{[(pn)^{\ell_0/2} + 1: n]}.
$$
For ease of notation, let $m = (pn)^{\ell_0/2}$.  We now find a block of substantial $\ell_2$ norm.  Observe that
\begin{equation}\label{eq:boundzl}
\|z_{\ell_0+1}\|_2 \leq (C_{\ref{lemma:verysparse}} K)^{-1}  \sqrt{m} \|z_{\ell_0+1}\|_\infty \leq \sqrt{2} (C_{\ref{lemma:verysparse}} K)^{-1} \|z_{\ell_0}\|_2.
\end{equation}
Since $x \in \mathcal{S}^{n-1}$ implies $\sum_{\ell+1}^{\ell_0+1} \|z_\ell\|_2^2 = 1$, we have
$$
\sum_{\ell=1}^{\ell_0} \|z_\ell\|_2^2 \geq 1 - 2(C_{\ref{lemma:verysparse}} K)^{-2}.
$$ 
 On the other hand, for $K \geq 1$, if $C_{\ref{lemma:verysparse}} > 2$, then $3 \sum_{\ell=1}^\infty (C_{\ref{lemma:verysparse}} K)^{-\ell} < 1$.  Therefore,
$$
\sum_{\ell=1}^{\ell_0} (C_{\ref{lemma:verysparse}} K)^{-2 \ell} < \sum_{\ell=1}^{\ell_0} \|z_\ell\|_2^2,
$$
from which one can deduce that there exists $\ell \leq \ell_0$ such that $\|z_\ell\|_2 \geq (C_{\ref{lemma:verysparse}} K)^{- \ell}$.  Let $\ell_*$ be the largest index with this property and define 
$u = \sum_{\ell=1}^{\ell_*} z_\ell$, $v = \sum_{\ell = \ell_*+1}^{\ell_0+1} z_\ell$.  We begin with the case $\ell_* < \ell_0$.  By the triangle inequality and (\ref{eq:boundzl}), we have
$$
\|v\|_2 \leq \sum_{\ell= \ell_*+1}^{\ell_0+1} \|z_\ell\|_2 \leq 2 \sqrt{2} (C_{\ref{lemma:verysparse}} K)^{-\ell_* + 1}.
$$ 
Let $\kappa = (pn)^{(\ell_*-1)/2}$.  Note that
$$
\kappa \leq (np)^{(\ell_0-1)/2} \leq \frac{1}{8p \sqrt{pn}}.
$$
We apply Lemma \ref{lemma:matrix} with this choice fo $\kappa$.  Divide the support of $u$ into $\sqrt{pn}$ blocks of size $\kappa$.  Define $L_{\ell_*} := \pi_x^{-1} ([1, (np)^{\ell_*/2}])$, where $\pi_x$ is the permutation arranging the coordinates of $x$ in decreasing order with respect to magnitude. For $s \in [(pn)^{1/2}]$, define $J_s:=\pi_x^{-1}([(s-1)\kappa+1, s\kappa])$, and set $J_s'=L_{\ell_*} \setminus J_s$.  Since $|J_s'| \le |L_{\ell_*}| = \kappa \sqrt{pn}$, we apply Lemma \ref{lemma:matrix} to get a set $\mathcal{A}$ with large probability, such that on $\mathcal{A}$, there exists subset of rows $I_s$ with $|I_s| \ge c_{\ref{lemma:matrix}} \kappa p n$ for all $s \in [\sqrt{pn}]$, such that for every $i \in I_s$, we have $|a_{i,j_0}|  \ge 1$ for only one index $j_0 \in J_s$ and $a_{i,j}=0$ for all $j \in J_s \cup J_s' \setminus \{j_0\}$. It can further be checked that $I_1,I_2,\cdots,I_{\sqrt{pn}}$ are disjoint subsets. Therefore, on set $\mathcal{A}$ for any $i \in I_s$,
  \[
    |((M_n - \lambda)u)_i|=|(M_n)_{i,{j_0}}u(j_0)|
    = |(M_n)_{i,j_0}| \cdot |u(j_0)| \ge |x(\pi^{-1}_x(s\kappa))|.
  \]
Here  we used  that $\pi_x$ is a non-increasing rearrangement. Now note that for $i \notin \text{supp}(u)$,
  \[
  ((M_n - \lambda)u)_i = (M_nu)_i, \text{ and } \text{supp}(u)= \kappa \sqrt{np} \ll c_{\ref{lemma:matrix}} \kappa np,
  \]
  as long as  $np \rightarrow \infty$. Therefore,
  \begin{align}
    \|(M_n - \lambda)u\|_2^2
    \ge \sum_{s=1}^{(pn)^{1/2}} \sum_{i \in I_s\setminus \supp(u)} \big( (M_n u)_i \big)^2
   & \ge \frac{c_{\ref{lemma:matrix}}pn}{2} \sum_{s=1}^{(pn)^{1/2}}\kappa (x(\pi_x^{-1}(s\kappa)))^2 \notag\\
    &\ge \frac{c_{\ref{lemma:matrix}}pn}{2} \sum_{k=(pn)^{(\ell_*-1)/2}}^{(pn)^{\ell_*/2}} (x(\pi_x^{-1}(k)))^2 \notag \\
    &= \frac{c_{\ref{lemma:matrix}}pn}{2} \|z_{\ell_*}\|_2^2  \ge \frac{c_{\ref{lemma:matrix}}pn}{2} \cdot ({C_{\ref{lemma:verysparse}}}K)^{-2 \ell_*},  \label{i: z_l}
\end{align}
  where the third inequality uses the monotonicity of the sequence $\{|x(\pi_x^{-1}(k))|\}_{k=1}^n$.
  Combining the above with the bound on $\|v\|_2$, on the set $\mathcal{A}$, we get that
  \begin{align*}
      \|(M_n - \lambda)x\|_2
      &\ge     \|(M_n - \lambda)u\|_2 - \|M_n - \lambda\| \cdot \|v\|_2  \\
      &\ge \sqrt{\frac{c_{\ref{lemma:matrix}}pn}{2} } ({C_{\ref{lemma:verysparse}}} K)^{- \ell_*}-  (K+R) \sqrt{pn} \cdot 2\sqrt{2} (C_{\ref{lemma:verysparse}} K)^{-(\ell_*+1)} \\
     & \ge \sqrt{pn}  (C_{\ref{lemma:verysparse}}' K)^{- \ell_*}\sqrt{pn},
  \end{align*}
  where the last inequality follows if the constants $C_{\ref{lemma:verysparse}}, C_{\ref{lemma:verysparse}}'$ are chosen large enough independently of $\ell_*$.

Now we consider the case when $\ell_*=\ell_0$.  Note that in this setting, using \eqref{i: z_l}, we have that
\[\|(M_n - \lambda u\|_2 \ge \sqrt{\frac{c_{\ref{lemma:matrix}}pn}{2} }  \|z_{\ell_0}\|_2,\]
and from \eqref{eq:boundzl}, we have $\|v\|_2=\|z_{\ell_0+1}\|_2 \le \sqrt{2} (C_{\ref{lemma:verysparse}}K)^{-1}  \|z_{\ell_0}\|_2$. Now proceeding similarly as before, on $\mathcal{A}$, we obtain that
\[\|(M_n - \lambda) x\|_2 \ge \sqrt{pn}  (C_{\ref{lemma:verysparse}} (K+R))^{- \ell_0}\sqrt{pn}.\]
Since by Lemma \ref{lemma:matrix}, $\P(\mathcal{A}) \ge 1 - \exp(-c_{\ref{lemma:matrix}} pn)$, the proof is complete.
\end{proof}
\subsection{Moderately Sparse Vectors}

\subsubsection{Small Ball Probability}
\begin{lemma}\label{lemma:smallball}
	Let $\delta_i$ be independent Bernoulli random variables taking value $1$ with probability $p$ and $\xi_i$ be independent random variables with mean $0$, variance $1$, and subgaussian moment bounded by $B$.   For a vector $x = (x_1, \dots, x_n) \in \mathbb{R}^n$, we have
	$$
	\mathcal{L}(\sum_{i=1}^n \delta_i \xi_i x_i, \frac{1}{4} \|x\|_2 \sqrt{p}) \leq 1 - \frac{c_{\ref{lemma:smallball}} p}{(\|x\|_\infty/\|x\|_2)^2 + p} 
	$$
\end{lemma}
\begin{proof}
	We can assume that $x$ is a unit vector.  We use the symmetrization technique to reduce the Levy function to a bound on the small ball probability around the origin.  Let $\delta'_1, \dots, \delta'_n$ and $\xi'_1, \dots, \xi'_n$ be independent copies of $\delta_1, \dots, \delta_n$ and $\xi_1, \dots, \xi_n$.  For any $r \in \mathbb{R}$,
	\begin{align} \label{eq:symmetrization}
	\mathbb{P}^2(|\sum_{i=1}^n \delta_i \xi_i x_i -r| \leq t) &= \mathbb{P} \left(| \sum_{i=1}^n \delta_i \xi_i x_i - r| \leq t \right) \mathbb{P}\left(| \sum_{i=1}^n \delta'_i \xi'_i x_i - r| \leq t \right) \\
	&\leq \mathbb{P}\left( | \sum_{i=1}^n (\delta_i \xi_i - \delta'_i \xi'_i) x_i| \leq 2t \right) 
	\end{align}
	Let $\zeta_i := \delta_i \xi_i - \delta'_i \xi'_i$ and $S:= \sum_{i=1}^n \zeta_i x_i$.  Observe $\mathbb{E} \zeta_i = 0$, $\mathbb{E} \zeta_i^2 = 2p$, $\mathbb{E}\zeta_i^3 = 0$ and $\mathbb{E}\zeta^4 = 2p \mathbb{E}\xi^4 + 6p^2 (\mathbb{E}\xi^2)^2 \leq Cp$ for some constant $C$ depending only on the subgaussian moment $B$.  $\mathbb{E}S^2 = \sum_{i=1}^n \mathbb{E} \zeta_i^2 \cdot x_i^2  = 2p$ and 
	\begin{align*}
	\mathbb{E} S^4 &=  \sum_{i=1}^n \mathbb{E} \zeta_i^4 \cdot x_i^4  + 3 \sum_{j \neq k} \mathbb{E}\zeta_j^2 x_j^2 \cdot \mathbb{E} \zeta_k^2 x_k^2 \leq C \|x\|_\infty^2 p + 12 p^2  
	\end{align*}
	for some constant $C'$ depending only on $B$.
	Thus by the Paley-Zygmund inequality, for $2t \leq \sqrt{2p}$,
	$$
	\mathbb{P}(|S| \leq 2t) \leq 1 - \frac{(\mathbb{E}S^2 - 4t^2)^2}{\mathbb{E} S^4} .
	$$
	Therefore,
	$$
	\mathbb{P}(|S| \leq \frac{1}{2}\sqrt{p} ) \leq 1 - \frac{c p}{C \|x\|_\infty^2 + p} .
	$$
	Combining this with inequality \ref{eq:symmetrization} yields
	$$
	\mathbb{P}(|\sum_{i=1}^n \delta_i \xi_i x_i -r| \leq \frac{1}{4}\sqrt{p}) \leq \sqrt{ 1 - \frac{c' p}{\|x\|_\infty^2 +p }}
	$$
	and setting $c_{\ref{lemma:smallball}} = c'/2$ yields the result. 
\end{proof}

\begin{lemma}
	For a random variable, $X$, with subgaussian moment bounded by $B$.  Then for any $k \in \mathbb{N}^{+}$,
	we have
	$$
	\mathbb{E}(|X|^k)^{1/k} \leq 2 B \sqrt k
	$$
\end{lemma}

\begin{lemma}\label{lemma:tensorization}
	Let $V_1, \dots, V_n$ be non-negative independent random variables such that $\P(V_i > 1) \geq q$, for all $i \in [n]$, and for some $q \in (0,1/2)$.  Then there exist constants $0 < c_{\ref{lemma:tensorization}}, c'_{\ref{lemma:tensorization}} < \infty$, such that
	$$
	\P \left( \sum_{j=1}^n V_j \leq \frac{c_{\ref{lemma:tensorization}} q n}{\log(1/q)} \right) \leq \exp(-c'_{\ref{lemma:tensorization}} n).
	$$
\end{lemma}

\begin{corollary} \label{cor:smallball}
	Let $M_n$ be a symmetric random matrix.  Then for any $\alpha > 1$, there exist $\beta, \gamma > 0$ such that for $x \in \R^{n}$ satisfying
	$$
	\|x\|_{\infty} / \|x\|_2 \leq \alpha \sqrt{p}, 
	$$
	we have
	$$
	\P \left( \|(M_n - \lambda) x\|_2, \beta \sqrt{pn } \|x\|_2 \right) \leq \exp(-\gamma n).
	$$
\end{corollary}
\begin{proof}
	Without loss of generality, we can assume that the coordinates of $x$ are organized by their magnitudes in decreasing order.  Let $n_0 = \lceil n/2 \rceil$.  
	$$
	M_n - \lambda I = 
	\begin{pmatrix} 
	A & B \\ 
	B^T & C 
	\end{pmatrix}, \quad
	x = \begin{pmatrix} y \\ z \end{pmatrix} 
	$$
	Then, 
	$$
	\|y\|_2 \geq \frac{1}{2} \|x\|_2
	$$
	so then $\|y\|_{\infty}/\|y\|_2 \leq 2 \alpha \sqrt p$.  Fix a $w \in \mathbb{R}^n$ and let 
	$V_j = \frac{16}{p \|y\|_2^2}((B^T y + C z)_j - w_j)^2$.  Also, by our assumptions,
	$$
	\frac{c_{\ref{lemma:smallball}}}{(\|y\|_\infty/ \|y\|_2)^2 + p} \geq \frac{c_{\ref{lemma:smallball}}}{4\alpha^2 + 1}
	$$
	By Lemmas \ref{lemma:smallball} and \ref{lemma:tensorization} we have
	$$
	\P \left( \|(M_n - \lambda) x\|_2, \beta \sqrt{pn } \|x\|_2 \right) \leq \P \left( \|B^T y + Cz\|_2, 2 \beta \sqrt{pn }  \|y\|_2 \right) 
	$$
\end{proof}

\subsection{Compressible Vectors}

\begin{proposition}
	$p^{-1} \leq M \leq c_{\ref{prop:compressible}} n$.
\end{proposition}

\begin{proof}
	We begin by diving $[n]$ into two roughly equal sets.  Let $n_0 = \lceil n/2 \rceil$.  We denote this decomposition by 
	$$
	M_n -\lambda I= 
	\begin{pmatrix} 
	A & B \\ 
	B^T & C 
	\end{pmatrix}, \quad
	x = \begin{pmatrix} y \\ z \end{pmatrix}, \quad
	u = \begin{pmatrix} v \\ w \end{pmatrix}.
	$$
	where $A$ is $n_0 \times n_0$ and $C$ is $n-n_0 \times n- n_0$.
	Thus, we have the following equivalence.
	$$
	\|(M_n - \lambda I )x\|_2^2 = \|Ay + Bz\|_2^2 + \|B^T y + C z \|_2^2.
	$$
	We condition on a realization of $A$ and $C$.  
	Let 
	$$
	W:= \text{Sparse}(M) \setminus \left( \text{Comp}((8p)^{-1}, \rho) \cup \text{Dom}( (8p)^{-1}, (C_{\ref{lemma:verysparse}}K)^{-1} \right)
	$$
	
	Denote $m = (8p)^{-1}$ so $m < M/2$.
	
	\textbf{Case I:}  Let's begin by assuming $p \geq \frac{1}{4} n^{-1/3}$.  In this regime, $\ell_0 = 1$ and so $\rho = (C'_{\ref{lemma:verysparse}} K)^{-2}$.  Observe that for $x \in V$, 
	$$
	\|x_{[m+1: M]}\|_{\infty} / \|x_{[m+1: M]} \|_2 \leq C_{\ref{lemma:verysparse}} K \sqrt{8p}
	$$
	Since, $x \notin \text{Comp}(m, \rho)$, $\|x_{[m+1:M]} \|_2 \geq \rho$.
	Thus, by Corollary \ref{cor:smallball}, 
	$$
	\P \left(\|(M_n - \lambda) x\|_2 \leq (C'_{\ref{lemma:verysparse}} K)^{-3}  \sqrt{pn} \|x_{[m+1:M]}\|_2 \right) \leq \exp(- c' n).
	$$
	Now we extend this bound to all vectors in $V$.  Define $\eps = (C'_{\ref{lemma:verysparse}} K)^{-4} \rho$.  There exists an $\eps$-net $\mathcal{N} \subset V$ of cardinality less than 
	$$
	\binom{n}{M} \left( \frac{3}{\eps} \right)^M \leq \exp \left( c_{\ref{prop:compressible}} n \log \left(\frac{3e}{c_{\ref{prop:compressible}} \eps} \right) \right) .
	$$
	Since $\lim_{x \rightarrow 0} x \log(1/x) = 0$ there exists a constant $c_{\ref{prop:compressible}}$ so that $c_{\ref{prop:compressible}} \log \left(\frac{3e}{c_{\ref{prop:compressible}} \eps} \right) \leq c'/2$.  Therefore, by the union bound,
	$$
	\P(\exists x \in \mathcal{N} : \|(M_n - \lambda I) x\|_2 \leq (C'_{\ref{lemma:verysparse}} K)^{-3} \sqrt{pn} \|x_{[m+1:M]} \|_2) \leq \exp(-(c'/2) n).
	$$
	We now extend this result to all of $W$.  Assume for all $x \in \mathcal{N}$,
	$$
	\|(M_n - \lambda I) x\|_2 \geq (C'_{\ref{lemma:verysparse}} K)^{-3} \sqrt{pn} \|x_{[m+1:M]} \|_2.
	$$
	Let $x' \in V$.  There exists a $x \in \mathcal{N}$ such that $\|x'- x\|_2 \leq \eps$.  
	We have 
	\begin{align*}
	\|(M_n - \lambda I) x'\|_2 &\geq \|(M_n - \lambda I) x\|_2 - \|M_n - \lambda I\| \|x - x'\|_2 \\
	&\geq  (C'_{\ref{lemma:verysparse}} K)^{-3} \sqrt{pn} \|x_{[m+1:M]} \|_2 - K \eps \\
	&\geq \frac{1}{2} (C'_{\ref{lemma:verysparse}} K)^{-3} \sqrt{pn} \rho
	\end{align*}
	
	\textbf{Case II:} We now tackle the remaining case where $\frac{C_{\ref{thm:main}} \log n}{n} \leq p \leq \frac{1}{4} n^{-1/3}$.  Let $I, J \subset [n]$ be disjoint sets such that $|I| = m$, $|J| = M-m$.  Let $\eps, \tau$ be positive numbers to be chosen later.  The sets 
	$$
	B_I := \{ u \in B_2^n : \text{supp}(u) \subset I \},
	$$
	and 
	$$
	R_J := \{u \in \mathcal{S}^{n-1} : \text{supp}(u) \subset J \text{ and } \|u\|_\infty \leq 4C_{\ref{prop:compressible}} K \sqrt{p} \}
	$$
	admit an $\eps$-net $\mathcal{N}_I \subset B_I$ and a $\tau$-net $\mathcal{N}_J \subset R_J$ of sizes
	$$
	|\mathcal{N}_I| \leq \left( \frac{3}{\eps} \right)^{|I|}
	$$
	and
	$$
	|\mathcal{N}_J| \leq \left( \frac{3}{\tau} \right)^{|J|}.
	$$
	Let $\mathcal{N}_0$ be an $\eps$-net in $[\rho/\sqrt{2}, 1] \subset \mathbb{R}$, and let
	$$
	\mathcal{M}_{IJ} :=  \{u + l w : u \in \mathcal{N}_I, w \in \mathcal{N}_J, l \in \mathcal{N}_0 \}
	$$ 
	and 
	$$
	\mathcal{M} := \left( \bigcup_{\substack{I \subset [n], \\ |I| = m}} \bigcup_{\substack{J \subset [n], \\ |J| = M-m, I \cap J = \emptyset}} \mathcal{M}_{IJ}  \right) .
	$$
	We now verify that this is an appropriate net for $W$.  Let $x \in W$ be decomposed as $x = u_x + v_x$ where $u_x = x_{[1:m]}$ and $v_x = x_{[m+1:M]}$.  Since $x \notin \text{Comp}(m, \rho) \cup \text{Dom}(m, (C_{\ref{lemma:verysparse}} K)^{-1})$, this implies that
	\begin{equation} \label{eq:dombound}
	\|v_x\|_2 \geq \rho \text{ and } \|v_x\|_\infty < C_{\ref{lemma:verysparse}} K \sqrt{8p} \|v_x\|_2.
	\end{equation}
	Choose $\bar{u} \in \mathcal{N}_I$, $\bar{v} \in \mathcal{N}_J$ and $l \in \mathcal{N}_0$ such that
	$$
	\|\bar{u} - u\| \leq \eps , \quad \| v_x / \|v_x\|_2 - \bar{v} \|_2 \leq \tau \quad \text{ and } \quad  |\|v_x\|_2 - l | \leq \eps .
	$$
	We can easily modify the net $\mathcal{M}$ so that $\mathcal{M} \subset W$ at the cost of adjusting $\eps$ and $\tau$ by a factor of 2.  Thus, by (\ref{eq:dombound}) we have for a fixed $\bar{x} \in \mathcal{M}$
	$$
	\P\left( (M_n - \lambda I) \bar{x} \leq (C'_{\ref{lemma:verysparse}} K)^{-3} \sqrt{pn} \|v_{\bar{x}}\|_2 \right) \leq e^{-c'n}.
	$$ 
	Now for $x \in W$, 
	$$
	\|(M_n - \lambda) x \|_2 \geq\|(M_n - \lambda) \bar{x}\|_2 - \|M_n - \lambda\| (\|u_x - u_{\bar{x}}\|_2 + \|v_x - v_{\bar{x}}\|_2).
	$$
	We observe that
	$$
	\|v_x - v_{\bar{x}}\|_2 \leq \left \| \frac{v_x}{\|v_x\|_2} -  \frac{v_{\bar{x}}}{\|v_{\bar{x}}\|_2} \right \|_2 \|v_{\bar{x}}\|_2 + \|v_x\|_2 \left | 1- \frac{\|v_{\bar{x}}\|_x}{\|v_x\|_2} \right| \leq 2 \tau \|v_{\bar{x}}\|_2 + 2 \eps .
	$$
	Therefore, letting $\mu' := (C_{\ref{lemma:verysparse}} K)^{-3}$,
	$$
	\|(M_n - \lambda) x \|_2 \geq \mu' \|v_{\bar{x}}\|_2 \sqrt{pn} - K \sqrt{pn} (3 \eps + 2 \tau \|v_{\bar{x}}\|_2).
	$$
	Setting $\eps := \frac{\mu' \rho}{12K}$ and $\tau = \frac{\mu'}{8 K}$ implies 
	$$\|(M_n - \lambda) x \|_2 \geq \frac{1}{2} \mu' \rho \sqrt{pn}.$$
	To take the union bound over all points in the net, we must obtain an upperbound on the size of the cardinality of the net. 
	$$
	|\mathcal{M}| \leq \binom{n}{m} \binom{n-m}{M-m} \left( \frac{3}{\eps} \right)^m \left( \frac{3}{\tau} \right)^{M-m} \frac{1}{\eps } .
	$$
	We first bound 
	$$
	\binom{n}{m} \binom{n-m}{M-m} \leq \binom{n}{m} \binom{n}{M} \leq \left( \frac{en}{m} \right)^m \left( \frac{en}{M} \right)^M \leq (8epn)^{(8p)^{-1}} \left( \frac{e}{c} \right)^{cn}.
	$$
	Thus, 
	$$
	|\mathcal{M}| \leq \left( \frac{288 e  K p n}{\mu' \rho} \right)^{(8p)^{-1}}  \left( \frac{24 e K}{c \mu'}\right)^{cn} .
	$$
	We claim that 
	$$
	\left( \frac{288 e  K p n}{\mu' \rho} \right)^{(8p)^{-1}} \leq \left( \frac{24 e K}{c \mu'}\right)^{cn}. 
	$$
	This reduces to the assertion that 
	$$
	p^{-1} \log \left( \frac{pn}{\rho} \right) = o(n)
	$$
	which is obvious by our assumption that $np \rightarrow \infty$ and $\ell_0 = o(np)$.  Finally,
	we conclude that 
	$$|\mathcal{M}| \leq \exp(-c'n/2)$$
	if we choose $c$ small enough since $\lim_{x \rightarrow 0} x \log(1/x) = 0$.  Therefore, a union bound concludes the proof.
\end{proof}

\section{Non-centered version of Proposition \ref{prop:compressible}} \label{appendix:noncentered}
To derive an analogue of Proposition \ref{prop:compressible}. 
We begin by diving $[n]$ into two roughly equal sets.  Let $n_0 = \lceil n/2 \rceil$.  We denote this decomposition by 
$$
A_n - p(J_n - I_n) = 
\begin{pmatrix} 
E & B \\ 
B^T & C 
\end{pmatrix}, \quad
x = \begin{pmatrix} y \\ z \end{pmatrix}.
$$
To lowerbound $\|(A_n - p(J_n - I_n))x \|_2^2$, it suffices to lower bound $\|E y + B z\|_2^2$.
For very sparse vectors, we can use the sign-matching argument from section \ref{lemma:verysparse} after conditioning on a realization of $E$.  For moderately sparse vectors,  the L\'evy concentration argument is insensitive to shifts and for the net argument, we add an extra net over the low-dimensional image as in Section \ref{section:randomgraph}.  We omit the details. 

\end{document}